\documentclass[12pt]{article}
\usepackage{amsmath,latexsym,amsxtra,enumerate,hyperref,pdfsync,color}
\usepackage[active]{srcltx}
\usepackage{bbm}
\usepackage{amsthm}
\usepackage[english]{babel}

\usepackage{times, amsfonts, amssymb, amsmath, amsthm, amscd, graphicx, stmaryrd}

\newcommand{\D}{\mathcal{D}}

\newcommand{\R}{\mathbb{R}}

\newcommand{\N}{\mathbb{N}}

\newcommand{\E}{\mathbb{E}}

\newcommand{\one}{{\rm 1\hspace{-0.11cm}I}}

\newcommand{\Ltwo}{\mathrm{L}_2}

\theoremstyle{plain}
\newtheorem{thm}{Theorem}[section]
\newtheorem{lem}[thm]{Lemma}

\newtheorem{cor}[thm]{Corollary}

\theoremstyle{definition}

\newtheorem{rem}[thm]{Remark}

\allowdisplaybreaks

\newcommand{\equa}{\begin{eqnarray*}}
\newcommand{\tion}{\end{eqnarray*}}
\newcommand{\equal}{\begin{eqnarray}}
\newcommand{\tionl}{\end{eqnarray}}
\newcommand{\m}{\mathbbm{m}}

\makeatletter
\def\timenow{\@tempcnta\time
\@tempcntb\@tempcnta
\divide\@tempcntb60
\ifnum10>\@tempcntb0\fi\number\@tempcntb
:\multiply\@tempcntb60
\advance\@tempcnta-\@tempcntb
\ifnum10>\@tempcnta0\fi\number\@tempcnta}
\makeatother

\begin{document}

%\newlength{\tmpmargin}
%\setlength{\tmpmargin}{\evensidemargin}
%\setlength{\evensidemargin}{\oddsidemargin}
%\setlength{\oddsidemargin}{\tmpmargin}

%\selectlanguage{english}

\title{Functionals of a L\'evy Process on Canonical and Generic Probability Spaces}

\author{Alexander Steinicke\\ \\
        Department of Mathematics  \\
        University of Innsbruck \\
        Technikerstra\ss e 19a \\
        A-6020 Innsbruck \\
        Austria \\
        alexander.steinicke@uibk.ac.at   \\ \\}

\maketitle

%Current version: \today, \timenow

\begin{abstract}
We develop an approach to Malliavin calculus for L\'evy processes from the perspective of expressing a random variable $Y$ by a functional $F$ mapping from the Skorohod space of c\`adl\`ag functions to $\R$, such that $Y=F(X)$ where $X$ denotes the L\'evy process. We also present a chain-rule-type application for random variables of the form $f(\omega,Y(\omega))$.

 An important tool for these results is a technique which allows us to transfer identities proved on the canonical probability space (in the sense of Sol\'e et al.) associated to a L\'evy process with triplet $(\gamma,\sigma,\nu)$ to an arbitrary probability space $(\Omega,\mathcal{F},\mathbb{P})$ which carries a L\'evy process with the same triplet.
\end{abstract}
%%%%%%%%%%%%%%%%%%%%%%%%%%%%%%%%%%%%%%%%%%%%%%%%%%%%%
\section{Introduction}
There exist various approaches for Malliavin calculus with jumps. One of the first works in this field were given by Bismut \cite{bismut} and by Bichteler et al. \cite{bicht}. Recently, the approaches how to define a Malliavin derivative for jump processes split up into two kinds. One way is to consider difference operators, as for example Nualart and Vives \cite{nualviv}, Picard \cite {picard}, Ishikawa and Kunita \cite{ishikawa}, Sol\'e et al. \cite{suv}, Applebaum \cite{apple} and Di Nunno et al. \cite{dinun} do. The other way is to consider ''true`` derivative operators, where a chain rule is possible. Those are given for example by Carlen and Pardoux \cite{carlen}, Denis \cite{denis}, Decreusefond and Savy \cite{decr}, Bally et al. \cite{bally} and Kulik \cite{kulik}.\bigskip

In this paper, we will focus on the view of the Malliavin derivative as difference operator. Here Kiyoshi It\^o's chaos decomposition of random variables \cite{ito} from 1956 is an important tool for defining and working with objects depending on the driving L\'evy process. Many applications utilize the chaos decomposition to change via an isomorphism from the probabilistic view of an $\Ltwo$-space associated with a probability space $(\Omega,\mathcal{F},\mathbb{P})$ to a deterministic view of a direct sum of $\Ltwo$-spaces on $\left({[0,\infty[}\times\R\right)^n$. The chaos decomposition also induces a Fock space structure on the set of square integrable random variables.\bigskip

One particular application of this Fock space structure is a simple description of Malliavin derivatives and Skorohod integrals (Nualart \cite{nualartbook}, Sol\'e et al. \cite{suv}). Using this approach, it is evident that it can be developed for arbitrary probability spaces $(\Omega,\mathcal{F},\mathbb{P})$ where the $\sigma$-algebra $\mathcal{F}$ is generated by a L\'evy process. However, choosing $(\Omega,\mathcal{F},\mathbb{P})=(\Omega^c,\mathcal{F}^c,\mathbb{P}^c)$, the canonical probability space presented by Sol\'e et al. in \cite{suv2}, this type of Malliavin calculus can be introduced in a different manner which strongly relies on the structure of $(\Omega^c,\mathcal{F}^c,\mathbb{P}^c)$. For this special choice, the resulting objects (Malliavin derivatives, Skorohod integrals) coincide with those gained by the Fock space method in a natural way, which has also been shown in \cite{suv2}. Operating on this canonical space has many advantages. For example, a pointwise definition of the Malliavin derivative and pointwise calculations with this object are not a problem, whereas the other approach always demands consideration of the Malliavin derivative as an $\Ltwo$-object.\bigskip

The aim of this paper is to explore how far identities for random variables, for their Malliavin derivatives and for other random fields, which can be shown in one probability space (in particular the canonical space) remain valid in the case of an arbitrary other probability space. The main tools for such investigations are representations of random variables by It\^o's chaos expansion and by functionals of a L\'evy process. These concepts will be introduced in Sections \ref{set} and \ref{furep}. We use the chaos expansion to transform the stochastic setting into a deterministic one, which eventually results in a technique permitting us to transfer identities to different probability spaces. This method will be treated in Section \ref{comp}.\bigskip

Having established such a technique, we present an application in Section \ref{appl}: We compute the Malliavin derivative of random variables of the type $f\left(\cdot,Y(\cdot)\right)$ in the case of a pure-jump L\'evy-process, where $f\colon\Omega\times\R\to \R$ is a measurable function and $Y$ is a Malliavin differentiable random variable. Such functions appear for example when Malliavin differentiation is applied to a generator of a backward stochastic differential equation driven by a L\'evy-process (see for example \cite {elie}, \cite{DelongImk}).
%%%%%%%%%%%%%%%%%%%%%%%%%%%%%%%%%%%%%%%%%%%%%%%%%%%%%%%%%%%%%%%%%%%%%%%%%%%%%%%%%%%
\section{Setting}\label{set}
Throughout the whole paper let $X=\left(X_t\right)_{t\geq 0}$ be a L\'evy-process on a complete probability space $(\Omega,\mathcal{F},\mathbb{P})$
with L\'evy-triplet $(\gamma,\sigma,\nu)$. Assume furthermore that $\mathcal{F}$ is the completion of $\mathcal{F}^X$, the $\sigma$-algebra generated by the process $X$.\bigskip

The L\'evy-It\^o decompositon of $X$ can be written as
\begin{equation*}
X_t = \gamma t + \sigma W_t   +  \int_{{]0,t]}\times\left\{0<|x|\leq 1\right\}} x\tilde{N}(ds,dx)+\int_{{]0,t]}\times\left\{|x|> 1\right\}} x N(ds,dx),
\end{equation*}
where $\sigma\geq 0$, $W$ is a Brownian motion, $N$ is the Poisson random measure and $\tilde N$ the compensated Poisson random measure corresponding to $X$.\bigskip

Define the measure
\equa
\m(dt,dx) =(\lambda\otimes\mu) (dt,dx)
\tion
with
\[\mu(dx)=\sigma^2\delta_0(dx)+x^2\nu(dx)\]
and $\lambda$ denoting the Lebesgue measure. For a set $B\in\mathcal{B}\left({[0,\infty[}\times\R\right)$ with $\m(B)\!<\!\infty$ we define the martingale-valued random measure $M$ by
\equa
   M(B):=\sigma \left(dW\otimes\delta_0\right)(B)+\lim_{n\to\infty}\int_{B\cap\left({[0,\infty[}\times\{1/n<|x|<n\}\right)}x\tilde N(dt,dx),
\tion
where the limit is taken in the space $\Ltwo\left(\Omega,\mathcal{F},\mathbb{P}\right)$. It holds $$\E M(B)^2 = \int_B\m(dt,dx).$$

With respect to the measure $M$ one can define multiple stochastic Integrals $I_n(f_n)$ for $n\in\N$ and for functions $$f_n\in\Ltwo^n:=\Ltwo\left(\left({[0,\infty[}\otimes\R\right)^n,\mathcal{B}\left(\left({[0,\infty[}\otimes\R\right)^n\right),\m^{\otimes n}\right).$$
Every random variable $Y\in\Ltwo\left(\Omega,\mathcal{F},\mathbb{P}\right)$ has a representation
\begin{equation}\label{chaosdec}
Y=\sum_{n=0}^\infty I_n(f_n),\quad\mathbb{P}\text{-a.s}.
\end{equation}
with unique, symmetric integrands $f_n\in\Ltwo^n$. We will refer to (\ref{chaosdec}) as the chaos expansion of $Y$. For precise definitions and more properties of the chaos expansion and multiple integrals see \cite{ito} or \cite{suv}.\bigskip

We use the following notation:
\begin{itemize}
\item With $D({[0,\infty[})$ we denote the Skorohod space of c\`adl\`ag functions on the interval ${[0,\infty[}$ equipped with the Skorohod topology. The $\sigma$-algebra we assume on this space is the Borel $\sigma$-algebra for this topology which coincides with the $\sigma$-algebra generated by the family of coordinate projections $\left(p_t\colon D({[0,\infty[})\to \R,\ f\mapsto f(t),\ t\geq 0\right)$ (see \cite{EKurtz}, Chapter 3 for details).
\item For $n\in\N, n\geq 1$ let $\mathcal{C}^\infty_c(\R^n)$ be the space of infinitely differentiable functions on $\R^n$ with compact support.
\item The canonical probability space for the L\'evy-triplet $(\gamma,\sigma,\nu)$ in the sense of \cite{suv2} will be denoted by $(\Omega^c,\mathcal{F}^{W,J},\mathbb{P}^c)$. If we consider the completion of the $\sigma$-algebra $\mathcal{F}^{X^c}$ which is generated by the canonical L\'evy process on $(\Omega^c,\mathcal{F}^{W,J},\mathbb{P}^c)$, we will denote the resulting probability space by $(\Omega^c,\mathcal{F}^c,\mathbb{P}^c)$.
\end{itemize}
Since the canonical probability space $(\Omega^c,\mathcal{F}^c,\mathbb{P}^c)$ and the canonical L\'evy process $X^c$ play a major role in Section 5, we will briefly describe the construction here:

The space $(\Omega^c,\mathcal{F}^{W,J},\mathbb{P}^c)$ is the product space $\left(\Omega^W\!\!\times\Omega^J,\mathcal{F}^W\otimes\mathcal{F}^J,\mathbb{P}^W\!\!\otimes\mathbb{P}^J\right)$. The canonical space for the Brownian part which we denote by $(\Omega^W,\mathcal{F}^W,\mathbb{P}^W)$ is the space of continuous functions endowed with the $\sigma$-algebra generated by the topology of uniform convergence on compact sets, and $\mathbb{P}^W$ is the Wiener measure. The canonical space for the jump part $(\Omega^J,\mathcal{F}^J,\mathbb{P}^J)$ makes use of a partition $(S_k)_{k\geq 1}$ of $\mathbb{R}\setminus\{0\}$, with $0\leq\nu(S_k)<\infty$. Moreover, assume that we remove the $S_k$ from the sequence with $\nu(S_k)=0$ and correct the numbering. If there remains only a finite sequence, we are in the compound Poisson case and would not need the whole following construction. However, one may nevertheless continue with the whole procedure by taking then all the infinite products and sums that appear in connection with this sequence as finite. Moreover, without loss of generality, assume that $S_1=\{x\in\R : 1< |x|\}$. The first step is to establish the space for a compound Poisson process on the interval ${[0,T]}$: Define $(\Omega^k_T,\mathcal{F}^k_T,\mathbb{P}^k_T)$ by taking $\Omega^k_T=\bigcup_{n\geq 0}({[0,T]\times S_k})^n$, with the convention that $({[0,T]\times S_k})^0=\{\alpha\}$, where $\alpha$ denotes a distinguished element representing the empty sequence. The $\sigma$-algebra $\mathcal{F}^k_T$ we define as $\bigvee_{n\geq 0}\mathcal{B}\left(({[0,T]\times S_k})^n\right)$ and for $B=\bigcup_{n\geq 0}B_n$, $B_n \in \mathcal{B}\left(({[0,T]\times S_k})^n\right)$, we set $$\mathbb{P}^k_T(B)=e^{-\nu(S_k)T}\sum_{n=0}^\infty\frac{\nu(S_k)^n(\lambda\otimes Q_k)^{\otimes n}(B_n)}{n!},$$
with $Q_k=\nu(\cdot\cap S_k)/\nu(S_k)$ and $(\lambda\otimes Q_k)^0=\delta_\alpha$. This yields that the process
$$X^k_{T,t}(\omega)=\begin{cases}\sum_{j=1}^n x_j\one_{[0,t]}(t_j), & \text{ if }\omega=((t_1,x_1),\dotsc,(t_n,x_n))\\ 0, & \text{ if }\omega=\alpha\end{cases}$$
is compound Poisson with L\'evy measure $\nu(\cdot\cap S_k)$ on the interval ${[0,T]}$. The elements $\omega$ are thus lists containing jumping times and jump sizes of $X^k_{T,\cdot}$. By symmetrizing the $\sigma$-algebras $\mathcal{F}^k_T$ and by a projective limit construction, taking these spaces with $T=1,2,\dotsc$ one extends this approach to a probability space $(\Omega^k,\mathcal{F}^k,\mathbb{P}^k)$ for the case of the infinite time interval ${[0,\infty[}$. With this construction, $\omega$ (up to the empty sequence) can then be interpreted to be an infinite list $\left((t_1,x_1),(t_2,x_2),\dotsc\right)$. Again, one then defines the canonic compound Poisson process as
$$X^k_{t}(\omega)=\begin{cases}\sum_{j=1}^n x_j\one_{[0,t]}(t_j), & \text{ if }\omega=((t_1,x_1),(t_2,x_2),\dotsc)\\ 0, & \text{ if }\omega=\alpha\end{cases},$$
for a detailed description of $(\Omega^k,\mathcal{F}^k,\mathbb{P}^k)$ see \cite[Section 4]{suv2}. For the L\'evy triplet $(\gamma,\sigma,\nu)$ the canonical probability space for the jump part $(\Omega^J,\mathcal{F}^J,\mathbb{P}^J)$ is defined by
$$(\Omega^J,\mathcal{F}^J,\mathbb{P}^J)=\bigotimes_{k \geq 1}(\Omega^k,\mathcal{F}^k,\mathbb{P}^k),$$
and the jump part of the canonical L\'evy process is
\begin{equation}\label{canproc}
X^J_t=\lim_{n\to\infty}\sum_{k=2}^n\left(X^k_t(\omega^k)-t\int_{S_k}x\nu(dx)\right)+X^1_t(\omega^1),
\end{equation}
for $\omega=(\omega^k)_{k\geq 1}$, where the convergence is $\mathbb{P}^J$-a.s., uniform on $t\in{[0,T]}$ for any $T>0$. Furthermore, in \cite[Proposition 4.8.]{suv2} it is stated, that the operation $${[0,\infty[}\times(\R\setminus\{0\})\times\Omega^J\to\Omega^J,\quad\left((r,v),\omega\right)\mapsto\omega_{r,v},$$
is measurable, where
\begin{equation}\label{omegaext}
\omega_{r,v}=\left((r,v),(t_1,x_1),\dotsc\right)\quad\text{if}\quad\omega=\left((t_1,x_1),(t_2,x_2),\dotsc\right).
\end{equation}

%%%%%%%%%%%%%%%%%%%%%%%%%%%%%%%%%%%%%%%%%%%%
\section{Representation of processes by functionals}\label{furep}

Consider the following representation lemma for random variables:
\begin{lem}\label{replem}
Let $Y$ be a random variable on $(\Omega,\mathcal{F},\mathbb{P})$. Then there exists a measurable functional $F\colon D\left({[0,\infty[}\right)\to\R$, such that
\begin{equation*}
Y=F\left((X_t)_{t\geq 0}\right),\ \mathbb{P}\text{-a.s.}
\end{equation*}
\end{lem}
The assertion follows from the fact that the random variable $Y$ equals an $\mathcal{F}^X$-measurable random variable $\tilde Y$ up to a $\mathbb{P}$-null set and the factorization lemma stated in \cite[Lemma II.11.7]{bauer} or the one in \cite[Theorem 1.1.7]{yongzhou}. We should also mention \cite{delzeith}, where filtrated Borel spaces are treated. Note that the functional $F$ in Lemma \ref{replem} is in general only unique $\mathbb{P}_X$-a.s, where $\mathbb{P}_X$ denotes the pushforward measure for the map $m_X\colon\Omega\to D\left({[0,\infty[}\right), \omega\mapsto \left(X_t(\omega)\right)_{t\geq 0}$.\bigskip

This section is dedicated to an extension of Lemma \ref{replem} for random processes which depend on an additional parameter $u$ taken from a measurable space $(U,\mathcal{U})$, which will be crucial in the latter part of the paper. A functional representation in this sense can easily be found by the factorization lemmas in \cite{bauer}, \cite{yongzhou} again for $\mathcal{F}^X\otimes\mathcal{U}$-measurable processes $Y\colon \Omega \times U\to\R$. However, here we have in mind the functional representation of an $\mathcal{F}\otimes\mathcal{U}$-measurable process, thus investigating the completion of $\mathcal{F}^X$. Since the $\sigma$-algebra generated by an $\mathcal{F}\otimes\mathcal{U}$-measurable process does not have to be contained in $\mathcal{F}^X\otimes\mathcal{U}$, we cannot apply the factorizations given in \cite{bauer} and \cite{yongzhou}. To overcome this problem, we need the next Lemma:
\begin{lem}\label{paramreplem}
Let $(G,\mathcal{G})$ be a measurable space, $\mathcal{N}$ be a $\sigma$-ideal of $(G,2^G)$ (for the definition of a $\sigma$-ideal see \cite{bauer}) and let $Y$ be an $(\mathcal{\mathcal{G}\vee \mathcal{N}})\otimes\mathcal{U}$-measurable mapping $Y\colon G\times U\to\R$. Then there exists an $G\otimes\mathcal{U}$-measurable mapping $\tilde Y\colon G\times U\to \R$ such that there exists a set $N\in\mathcal{N}$ and for $x\notin N$ it holds that
\begin{equation}\label{indist}
Y(x,\cdot)=\tilde Y(x,\cdot).
\end{equation}
\end{lem}
\begin{rem}
 If in the situation of Lemma \ref{paramreplem} $\mathcal{N}$ is the set of $\mu$-null sets of a probability space $(G,\mathcal{G},\mu)$, then $\mathcal{G}\vee\mathcal{N}$ is the completion of $\mathcal{G}$. Equation (\ref{indist}) then means that the two processes $Y$ and $\tilde Y$ are indistinguishable. Also the case that $\mathcal{N}$ is the set of $\mu$-null sets of a $\sigma$-algebra which is also a domain of $\mu$ and contains $\mathcal{G}$ is covered by the lemma. This situation happens for example if augmentations of filtrations are considered.
\end{rem}
Taking $(G,\mathcal{G},\mu)=(\Omega,\mathcal{F}^X,\mathbb{P})$ and $\mathcal{N}$ to be the $\mathbb{P}$-null sets of $\Omega$, this Lemma directly implies the following Theorem \ref{repthm}, which we formulate before proving Lemma \ref{paramreplem}.
\begin{thm}\label{repthm}
Let $Y$ be an $\mathcal{F}\otimes\mathcal{U}$-measurable process $Y\colon\Omega\times U\to\R$. Then there exists a measurable functional $F\colon D\left({[0,\infty[}\right)\times U \to \R$ such that
\begin{equation*}
Y(\omega,\cdot)=F\left((X_t(\omega))_{t\geq 0},\cdot \right),
\end{equation*}
for almost all $\omega\in\Omega$.
\end{thm}

\begin{proof}[Proof of Lemma \ref{paramreplem}]
$ $\\
{\it 1st step:}\\
The $\sigma$-algebra $\mathcal{G}\vee \mathcal{N}$ can be expressed as $\mathcal{M}\!:=\!\left\{B\subseteq 2^G\!:\!\exists A\in \mathcal{G}\colon\! A\triangle B\in\mathcal{N}\right\}$.\\
{\it Proof:}\\
Since for all $A\in\mathcal{G}$ we have $A\triangle A=\emptyset\in\mathcal{N}$ we have that $\mathcal{G}$ is contained in $\mathcal{M}$. Because for all $N\in\mathcal{N}$ it holds that $\left(G\setminus N\right)\triangle G=N$, $\mathcal{M}$ also contains complements of sets in $\mathcal{N}$. By the formulas $$A\triangle B=\left(G\setminus A\right)\triangle\left(G\setminus B\right)$$ and $$\left(\bigcup_{n\in\N}A_n\right)\triangle\left(\bigcup_{n\in\N}B_n\right)\subseteq \bigcup_{n\in\N}\left(A_n\triangle B_n\right)$$ and using the property that $\mathcal{N}$ is a $\sigma$-ideal of $(G,2^G)$, it is easy to see that $\mathcal{M}$ is a $\sigma$-algebra. It remains to show that, if $\mathcal{H}$ is a $\sigma$-algebra containing $\mathcal{G}$ and $\mathcal{N}$, it also contains $\mathcal{M}$. To see this, we take $B\in\mathcal{M}$. Thus there is $A\in \mathcal{G}$ such that $A\triangle B\in\mathcal{N}$. Since $A\in\mathcal{H}$, we also get $(A\triangle B)\cup A=A\cup B\in\mathcal{H}$. By further elementary set operations it follows that $A\cap B, A\setminus B, B\setminus A$ are all in $\mathcal{H}$, as is then $B$.\bigskip

{\it 2nd step:}\\
Let $\mathcal{N}\otimes\mathcal{U}$ be defined as the smallest $\sigma$-ring which contains all products $N\times B$, $N\in\mathcal{N}$, $B\in \mathcal{U}$.
It holds that \begin{equation}\label{muelldeponie}
(\mathcal{G}\vee\mathcal{N})\otimes\mathcal{U}=(\mathcal{G}\otimes\mathcal{U})\vee(\mathcal{N}\otimes\mathcal{U})\subseteq(\mathcal{G}\otimes\mathcal{U})\vee\widehat{\mathcal{N}\otimes\mathcal{U}},
\end{equation}
where $\widehat{\mathcal{N}\otimes\mathcal{U}}$ denotes the set of all subsets of sets in $\mathcal{N}\otimes\mathcal{U}$, which is a $\sigma$-ideal of $(G\times U,2^{G\times U})$. Hence, the inclusion in equation (\ref{muelldeponie}) is trivial.\\
{\it Proof:}\\
We consider the generating sets: First, notice that
$$\left(\mathcal{G}\vee\mathcal{N}\right)\otimes\mathcal{U}=\sigma\left(A\times B: A\in \mathcal{G}\vee\mathcal{N}, B\in \mathcal{U}\right).$$
We show that each generator set $A\times B$ is contained in $(\mathcal{G}\otimes\mathcal{U})\vee(\mathcal{N}\otimes\mathcal{U})$. Indeed, since cartesian products are distributive over (arbitrary) unions and complements, it follows that each $A\times B$ as above is contained in the smallest $\sigma$-ring generated by the set of products $P:=\left\{C\times B : C\in \mathcal{G}\cup\mathcal{N}\right\}$, where we fixed $B\in\mathcal{U}$. The set of products $P$ and the $\sigma$-ring it generates are contained again in $(\mathcal{G}\otimes\mathcal{U})\vee(\mathcal{N}\otimes\mathcal{U})$ which implies $$(\mathcal{G}\vee\mathcal{N})\otimes\mathcal{U}\subseteq(\mathcal{G}\otimes\mathcal{U})\vee(\mathcal{N}\otimes\mathcal{U}).$$
On the other hand, $\mathcal{G}\times\mathcal{U}$ and $\mathcal{N}\times\mathcal{U}$ are both subsets of $\left(\mathcal{G}\vee\mathcal{N}\right)\otimes\mathcal{U}$ which yields the other inclusion.\bigskip

{\it 3rd step:}\\
Let $\mathcal{A}$ and $\mathcal{B}$ be $\sigma$-rings. Then for all sets $M$ in the product $\sigma$-ring $\mathcal{A}\otimes\mathcal{B}$ there is $A\in\mathcal{A}$ and $B\in\mathcal{B}$ such that $M\subseteq A\times B$. This can instantly be applied to the $\sigma$-ring $\mathcal{N}\otimes\mathcal{U}$ and the assertion then also holds for all subsets of sets in $\mathcal{N}\otimes\mathcal{U}$, which means that it holds for all sets in $\widehat{\mathcal{N}\otimes\mathcal{U}}$.\\
{\it Proof:}\\
This proof is an application of the ''principle of appropriate sets`` (see Shiryaev, \cite[\S 2. Theorem 1]{shiraev}). We define the set system
$$\mathcal{K}:=\left\{M\in \mathcal{A}\otimes\mathcal{B}: \exists (A,B)\in\mathcal{A}\times \mathcal{B}: M\subseteq A\times B\right\}.$$
Clearly, $\mathcal{K}$ contains all products of the form $A_0\times B_0$, $A_0\in\mathcal{A}$, $B_0\in \mathcal{B}$. We now show that $\mathcal{K}$ is a $\sigma$-ring:

Take $M_1, M_2\in \mathcal{K}$. Then $M_1$ is contained in a product $A_1\times B_1$. But then also $M_1\setminus M_2\subset M_1$ is contained in $A_1\times B_1$. Thus $M_1\setminus M_2\in \mathcal{K}$.

Now take a sequence $(M_n)_{n\in \N}$ in $\mathcal{K}$. So there are products $A_n\times B_n$ containing $M_n$. Then the inclusion
$$\bigcup_{n\in\N}M_n\subseteq\bigcup_{n\in\N}(A_n\times B_n)\subseteq\left(\bigcup_{n\in\N} A_n\right)\times\left(\bigcup_{n\in\N}B_n\right),$$
shows the $\sigma$-property of $\mathcal{K}$. Since $\mathcal{K}$ contains $\left\{A\times B: A\in \mathcal{A}, B\in \mathcal{B}\right\}$, $\mathcal{K}$ equals $\mathcal{A}\otimes\mathcal{B}$ and the step is proven.\bigskip

{\it 4th step:}\\
Let $N$ be a set in $\widehat{\mathcal{N}\otimes \mathcal{U}}$. Then the characteristic function $\one_N$ is the zero function in $u\in U$ up to a set in $\mathcal{N}$, i.e.
\begin{equation*}
\left\{x\in G : \exists u\in U :\one_N(x,u)\neq 0\right\}\in\mathcal{N}.
\end{equation*}

{\it Proof:}\bigskip

From the 3rd step it follows that for $N\in \widehat{\mathcal{N}\otimes \mathcal{U}}$, there is a product $N_1\times B_1$ with $N_1\in \mathcal{N}$ and $B_1\in \mathcal{U}$ such that $N\subseteq N_1\times B_1$. We get
\begin{equation*}
\left\{x\!\in G : \!\exists u\!\in U :\one_N(x,u)\neq 0\right\}\!\subseteq\!\left\{x\!\in G : \exists u\!\in U :\one_{N_1\times B_1}(x,u)\neq 0\right\}\!=\!N_1.
\end{equation*}

{\it 5th step:}\\
Let $Y\colon G\times U\to\R$ be a $(\mathcal{G}\vee\mathcal{N})\otimes\mathcal{U}$-measurable step-function,
$$Y=\sum_{k=1}^n\alpha_k\one_{A_k}.$$
Then there are sets $\check A_1,\dotsc,\check A_n\in \mathcal{G}\otimes\mathcal{U}$ and $N_1,\dotsc, N_{2n}\in \widehat{\mathcal{N}\otimes\mathcal{U}}$ such that
$$Y=\sum_{k=1}^n\alpha_k\one_{\check A_k}+\sum_{k=1}^{2n}\beta_k\one_{N_k}.$$

{\it Proof:}\\
By the 2nd step, the $A_k$ are $(\mathcal{G}\otimes\mathcal{U})\vee\widehat{\mathcal{N}\otimes\mathcal{U}}$-measurable. Thus, by the 1st step, there exists $\check A_k\in\mathcal{G}\otimes\mathcal{U}$, such that $A_k\triangle \check A_k\in \widehat{\mathcal{N}\otimes\mathcal{U}}$. Therefore the characteristic functions
 $\one_{A_k}$ have the form $\one_{A_k}=\one_{\check A_k} +\one_{A_k\setminus \check A_k}-\one_{\check A_k\setminus A_k}$, where both appearing set differences are elements of $\widehat{\mathcal{N}\otimes\mathcal{U}}$ for $k=1,\dotsc,n$. Therefore we can represent $Y$ by
\begin{equation*}
\sum_{k=1}^n\alpha_k\one_{A_k}=\sum_{k=1}^n\alpha_k\one_{\check A_k}+\sum_{k=1}^n\alpha_k\one_{A_k\setminus\check A_k}-\sum_{k=1}^n\alpha_k\one_{\check A_k\setminus A_k},
\end{equation*}
so taking $N_k=A_k\!\setminus\! \check A_k, k=1,\dotsc,n$ and $N_k=\check A_k\!\setminus\! A_k, k=n+1,\dotsc,2n$ yields the assertion.\bigskip

{\it Completion of the proof of Lemma \ref{paramreplem}}:\\
By the 4th and 5th step, we see that the assertion of Lemma \ref{paramreplem} already holds for measurable step-functions. Now, for an arbitrary $(\mathcal{G}\vee\mathcal{N})\otimes\mathcal{U}$-measurable function, we know that there exists a sequence $(Y_n)_{n\in\N}$ of measurable step-functions such that for all $(x,u)\in G\times U$, $Y_n(x,u)$ converges to $Y(x,u)$. Moreover, the 5th step implies that out of this sequence, one can construct a sequence $(\tilde Y_n)_{n\in \N}$ of $\mathcal{G}\otimes\mathcal{U}$-measurable functions, such that $$\tilde Y_n(x,\cdot)=Y_n(x,\cdot )$$
on $G\setminus N_n$, where $N_n\in\mathcal{N}$. Thus, on the set $G\setminus\left(\bigcup_{n\in \N}N_n\right)$ we have equality of the sequences of functions $$\left(Y_n(x,\cdot)\right)_{n\in \N}=\left(\tilde Y_n(x,\cdot)\right)_{n\in \N}.$$
Therefore we also know that $$\left\{x\in G: \forall u\in U :\left(\tilde Y_n(x,u)\right)_{n\in \N}\text{ converges}\right\}\!\supseteq G\setminus\left(\bigcup_{n\in \N}N_n\right),$$
and on this set it holds that $\lim_{n\to\infty}\tilde Y_n(x,u)=Y(x,u)$ for all $u\in U$. Define $N:=\bigcup_{n\in \N}N_n\in \mathcal{N}$. Then the function
\begin{equation*}
\tilde Y(x,u):=\lim_{n\to\infty}\tilde Y_n(x,u),
\end{equation*}
defined on $(G\setminus N)\times U$, is measurable with respect to the restricted $\sigma$-algebra $\left\{A\cap\left((G\setminus N)\times U\right) : A\in \mathcal{G}\otimes\mathcal{U}\right\}$ because all $\tilde Y_n\!\mid_{(G\setminus N)\times U}$ are. Thus, in view of \cite{shortt}, we can extend $\tilde Y$ to a $\mathcal{G}\otimes\mathcal{U}$- measurable mapping on $G\times U$. The resulting function satisfies the assertion of the lemma.
\end{proof}
%%%%%%%%%%%%%%%%%%%%%%%%%%%%%%%%%%%%%%%%
\section{Comparison of functionals of L\'evy processes on different probability spaces}\label{comp}

In this section we assume $\left(\Omega_1,\mathcal{F}_1,\mathbb{P}_1\right), \left(\Omega_2,\mathcal{F}_2,\mathbb{P}_2\right)$ to be probability spaces with L\'evy processes $X^i=(X^i_t)_{t\geq 0}$, $X^i_t\colon \Omega_i\to \R$, such that $X^i$ corresponds to a given L\'evy triplet $(\gamma,\sigma,\nu)$
for $i=1,2$. Furthermore, assume that $\mathcal{F}_i$ is the completion of the sigma algebra generated by $X^i$.
For the processes $X^1, X^2$, we get martingale valued measures $M^1,M^2$ and also families of multiple stochastic integrals $\left(I^1_n(f_n)\right)_{n\in\N}, \left(I^2_n(f_n)\right)_{n\in\N}$, respectively.

The following theorem states that random variables gained by applying the same functional to L\'evy processes which are defined on different probability spaces have the same integrands in the chaos expansion:
\begin{thm}\label{chaoseq}
Let $Y_1\in\Ltwo\left(\Omega_1,\mathcal{F}_1,\mathbb{P}_1\right)$, $Y_2 \in \Ltwo\left(\Omega_2,\mathcal{F}_2,\mathbb{P}_2\right)$ be random variables.
Suppose that $Y_1,Y_2$ have chaos decompositions
\begin{equation*}
Y_1=\sum_{n=0}^\infty I^1_n(f_n),\ \mathbb{P}_1\text{-a.s.},\quad Y_2=\sum_{n=0}^\infty I^2_n(g_n),\ \mathbb{P}_2\text{-a.s.}
\end{equation*}
with $f_n,g_n$ being symmetric functions in $\Ltwo^n$.

Assume that there is a measurable functional $F\colon D\left({[0,\infty[}\right)\to\R$
such that $Y_i=F\left((X^i_t)_{t\geq 0}\right)$, $\mathbb{P}_i$-a.s. for $i=1,2$. Then for all $n\in\N$ it holds that $$f_n=g_n,\quad\m^{\otimes n}\text{-a.e.}$$
\end{thm}
\begin{proof}
The random variables of the form $\psi\left(X^1_{t_1},\dotsm,X^1_{t_m}\right)$ with $t_1,\dotsm,t_m\geq 0$, $\psi \in \mathcal{C}^\infty_c(\R^m)$, $m\in\N$, are dense in $\Ltwo\left(\Omega_1,\mathcal{F}_1,\mathbb{P}_1\right)$ (see \cite[Corollary 4.1]{geissc}). Thus we can approximate $Y_1$ by a sequence of the type $\left(\psi_n\left(X^1_{t^n_1},\dotsm,X^1_{t^n_{m_n}}\right)\right)_{n\in\N}$, $\psi_n\in \mathcal{C}^\infty_c(\R^{m_n})$:
\begin{equation*}
\lim_{n\to\infty}\E_1 \left|Y_1-\psi_n\left(X^1_{t^n_1},\dotsm,X^1_{t^n_{m_n}}\right)\right|^2=0,
\end{equation*}
with $\E_1$ being the expectation with respect to $\mathbb{P}_1$.
By assumption we get
\begin{equation}\label{equaldistr}
\E_1 \left|Y_1-\psi_n\left(X^1_{t^n_1},\dotsm,X^1_{t^n_{m_n}}\right)\right|^2=\E_1 \left|F\left((X^1_t)_{t\geq 0}\right)-\psi_n\left(X^1_{t^n_1},\dotsm,X^1_{t^n_{m_n}}\right)\right|^2,
\end{equation}
and since the mappings $m_X^i\colon\Omega_i\to D\left({[0,\infty[}\right), \omega\mapsto\left(X^i_t(\omega)\right)_{t\geq 0}$ yield the same pushforward measures for $i=1,2$ the distributions of
$$\left|F\!\left((X^1_t)_{t\geq 0}\right)\!-\!\psi_n\! \left(X^1_{t^n_1},\dotsm,X^1_{t^n_{m_n}}\right)\right|^2\!\text{ and }\!\left|F\!\left((X^2_t)_{t\geq 0}\right)\!-\!\psi_n\!\left(X^2_{t^n_1},\dotsm,X^2_{t^n_{m_n}}\right)\right|^2$$ coincide.
Hence $Y_2$ can be approximated in the $\Ltwo$-norm by the same sequence of smooth random variables on the probability space $\left(\Omega_2,\mathcal{F}_2,\mathbb{P}_2\right)$.

The integrands in the chaos expansion of $\psi_n\left(X^i_{t^n_1},\dotsm,X^i_{t^n_{m_n}}\right)$, $i=1,2$ can be computed explicitly (see \cite[below Lemma 3.1]{geissc}) and are the same, independent of the choice of the probability space. Thus, taking $\Ltwo$-limits, the integrands of the chaos expansions of $Y_1$ and $Y_2$ are equal as well.
\end{proof}

The previous theorem can be extended to square integrable random fields:
\begin{cor}
Let $(E,\mathcal{E},\rho)$ be a $\sigma$-finite measure space and let $$C^1\in\Ltwo\left(\Omega_1\times E,\mathcal{F}_1\otimes\mathcal{E},\mathbb{P}_1\otimes\rho\right),$$ $$C^2\in\Ltwo\left(\Omega_2\times E,\mathcal{F}_2\otimes \mathcal{E},\mathbb{P}_2\otimes\rho\right)$$
and suppose that these random fields have chaos decompositions
\begin{equation*}
C^1=\sum_{n=0}^\infty I^1_n(f_n),\ \mathbb{P}_1\otimes\rho\text{-a.e.},\quad C^2=\sum_{n=0}^\infty I^2_n(g_n),\ \mathbb{P}_2\otimes\rho\text{-a.e.}
\end{equation*}
for $f_n,g_n$ being functions in $\Ltwo(E,\mathcal{E},\rho)\hat\otimes\Ltwo^{n}$ which are symmetric in the last $n$ variables, where '$\hat \otimes$' denotes the Hilbert space tensor product.

Assume that for $\rho$-almost all $e\in E$ there are functionals $$F_{e}\colon D\left({[0,\infty[}\right)\to\R$$
such that $C^i(e)=F_{e}\left((X^i_t)_{t\geq 0}\right)$, $\mathbb{P}_i$-a.s. for $i=1,2$. Then for all $n\in\N$ it holds $f_n=g_n$, $\rho\otimes\m^{\otimes n}$-a.e.
\end{cor}
\begin{proof}
By the theorem before, we know that for $\rho$-almost every $e\in E$ we have that
\begin{equation*}
\left\|f_n\left(e,\cdot\right)-g_n\left(e,\cdot\right)\right\|_{\Ltwo^n}=0
\end{equation*}
and therefore also $\left\|f_n-g_n\right\|_{\Ltwo(E,\mathcal{E},\rho)\hat\otimes\Ltwo^{n}}=0$.
\end{proof}
For later use we formulate the corollary above in the case of $(E,\mathcal{E},\rho)=({[0,\infty[}\times\R,\mathcal{B}\left({[0,\infty[}\times\R\right),\m)$, which is used in Section \ref{appl}:
\begin{cor}\label{chaosequalcor}
Let $$C^1\in\Ltwo\left(\Omega_1\times{[0,\infty[}\times\R,\mathcal{F}_1\otimes\mathcal{B}\left({[0,\infty[}\times\R\right),\mathbb{P}_1\otimes\m\right),$$ $$C^2\in\Ltwo\left(\Omega_2\times{[0,\infty[}\times\R,\mathcal{F}_2\otimes\mathcal{B}\left({[0,\infty[}\times\R\right),\mathbb{P}_2\otimes\m\right)$$
and suppose that these random fields have chaos decompositions
\begin{equation*}
C^1=\sum_{n=0}^\infty I^1_n(f_n),\ \mathbb{P}_1\otimes\m\text{-a.e.},\quad C^2=\sum_{n=0}^\infty I^2_n(g_n),\ \mathbb{P}_2\otimes\m\text{-a.e.}
\end{equation*}
for $f_n,g_n$ being functions in $\Ltwo^{n+1}$ which are symmetric in the last $n$ variables.

Assume that for $\m$-almost all $(r,v)\in{[0,\infty[}\times\R$ there are functionals $$F_{r,v}\colon D\left({[0,\infty[}\right)\to\R$$
such that $C^i(r,v)=F_{r,v}\left((X^i_t)_{t\geq 0}\right)$, $\mathbb{P}_i$-a.s. for $i=1,2$. Then for all $n\in\N$ it holds $f_n=g_n$, $\m^{\otimes n+1}$-a.e.
\end{cor}

%%%%%%%%%%%%%%%%%%%%%%%%%%%%%%%%%%%%%%%%%%%%%%%%%%%%%%%%%%%%%%%%%%%%%%%%%%%%%%%
\section{Application to Malliavin calculus}\label{appl}

In this section we assume the probability space $\left(\Omega^c,\mathcal{F}^c,\mathbb{P}^c\right)$ to be the canonical probability space for a L\'evy process $X$ with triplet $(\gamma,0,\nu)$. Thus we work in the setting of a pure jump process from now on.\bigskip

There are various methods to define a Malliavin derivative for the jump part of a L\'evy process. We focus on three particular approaches, which yield the Malliavin derivative seen as a difference operator. One is, given the space $\left(\Omega^c,\mathcal{F}^c,\mathbb{P}^c\right)$, to define an operator
\begin{equation*}
\Psi_{r,v}Y(\omega)=\frac{Y(\omega_{r,v})-Y(\omega)}{v},\quad v\neq 0,
\end{equation*}
where $\omega_{r,v}$ is defined in (\ref{omegaext}) (see also \cite{suv2}) and we can think of it as the path $\omega$ where at time $r$ a jump of the size $v$ is added.

Another approach is to define the Malliavin derivative via chaos expansions, developed for example in \cite{suv}. This approach can be applied to arbitrary probability spaces $\left(\Omega,\mathcal{F},\mathbb{P}\right)$ in the sense of Section \ref{set}: Define
$$\mathbb{D}_{1,2}(\Omega):=\left\{Y\in\Ltwo\left(\Omega,\mathcal{F},\mathbb{P}\right)\colon \left\|Y\right\|^2_{\mathbb{D}_{1,2}}=\sum_{n=0}^\infty (n+1)!\left\|f_n\right\|^2_{\Ltwo^n}<\infty\right\}$$
and for a random variable $Y\in \mathbb{D}_{1,2}(\Omega)$ having chaos decomposition $$Y=\sum_{n=0}^\infty I_n(f_n),$$ let
$$\D_{r,v}Y:=\sum_{n=1}^\infty I_{n-1}\left(f_n((r,v),\cdot)\right)$$
for $\mathbb{P}\otimes\m-a.a.\ (\omega,r,v)\in\Omega\times{[0,\infty[}\times\left(\R\setminus\{0\}\right)$. The Malliavin derivative $\D Y$ is then an object in the space $$\Ltwo\left(\mathbb{P}\otimes\m\right)=\Ltwo\left(\Omega\times\!{[0,\infty[}\times\!\left(\R\!\setminus\!\{0\}\right),\mathcal{F}\otimes\mathcal{B}\left({[0,\infty[}\times\!\left(\R\!\setminus\!\{0\}\right)\right),\mathbb{P}\otimes\m\right).$$
A third method \cite{geissc} is to define an operator working on a dense set (in $\mathbb{D}_{1,2}(\Omega)$) of 'smooth' random variables $f\left(X_{t_1},\dotsm,X_{t_n}\right)$, $f\in\mathcal{C}^\infty_c\left(\R^n\right)$, $n\in\N$, defined by
\begin{equation*}
\begin{split}
\Phi_{r,v}&f\left(X_{t_1},\dotsm,X_{t_n}\right)\\
&:=\frac{f\left(X_{t_1}+v\one_{[0,t_1]}(r),\dotsm,X_{t_n}+\one_{[0,t_n]}(r)\right)-f\left(X_{t_1},\dotsm,X_{t_n}\right)}{v},
\end{split}
\end{equation*}
and extend it by linearity and continuity to random variables in $\mathbb{D}_{1,2}(\Omega)$. All the resulting objects of these methods coincide on the space $\mathbb{D}_{1,2}(\Omega^c)$, the latter ones also on $\mathbb{D}_{1,2}(\Omega)$ for an arbitrary $\left(\Omega,\mathcal{F},\mathbb{P}\right)$ with the conditions of Section \ref{set} (see \cite{suv2}, \cite{geissc}). Despite slight abuse of notation, we stick to '$\D$' denoting the derivative operator in all cases.
One aim of this section is to calculate the Malliavin derivative of a functional of a L\'evy process with methods developed in the section before. We then prove a 'chain rule-type' formula for random variables of the type $f(\cdot,Y(\cdot))$.
\begin{thm}\label{diffthm}
Assume $Y=F\left((X_t)_{t\geq 0}\right)\in\mathbb{D}_{1,2}(\Omega) $, $F\colon D\left({[0,\infty[}\right)\to\R$ measurable. Then for $\mathbb{P}\otimes\m$-a.a. $(\omega,r,v)\in\Omega\times{[0,\infty[}\times\left(\R\setminus\{0\}\right)$ it holds that
\begin{equation}\label{chain}
\D_{r,v}Y=\frac{F\left((X_t+v\one_{[0,t]}(r))_{t\geq 0}\right)-F\left((X_t)_{t\geq 0}\right)}{v}.
\end{equation}
\end{thm}
\begin{proof}
Define $Y^c:=F\left(\left(X^c_t\right)_{t\geq 0}\right)$. Thus $Y^c\in \mathbb{D}_{1,2}(\Omega^c)$ since the integrands of the chaos expansions of $Y$ and $Y^c$ are equal in $\Ltwo^n$ according to Theorem \ref{chaoseq}. Moreover, from the equality of $\D$ and $\Psi$ on $\mathbb{D}_{1,2}(\Omega^c)$ it follows that
\begin{equation*}
\D_{r,v}Y^c(\omega)=\Psi_{r,v}Y^c(\omega)=\frac{F\left(\left(X^c_t(\omega_{r,v})\right)_{t\geq 0}\right)-F\left(\left(X^c_t(\omega)\right)_{t\geq 0}\right)}{v},
\end{equation*}
$\mathbb{P}^c\otimes\m$-a.e.
By definition of $X^c$ and $\omega_{r,v}$ given in (\ref{canproc}) and (\ref{omegaext}) or \cite[Section 4.3. and 4.4.]{suv2}, we know that $X^c_t(\omega_{r,v})=X^c_t(\omega)+v\one_{[0,t]}(r)$, implying
\begin{equation*}
\D_{r,v}Y^c(\omega)=\frac{F\left(\left(X^c_t(\omega)+v\one_{[0,t]}(r)\right)_{t\geq 0}\right)-F\left(\left(X^c_t(\omega)\right)_{t\geq 0}\right)}{v},
\end{equation*}
$\mathbb{P}^c\otimes\m$-a.e.
Because translations $D({[0,\infty[})\to D({[0,\infty[}),\ h\mapsto h+k,\ k\in D({[0,\infty[})$, are measurable functions, we can use Corollary \ref{chaosequalcor}, which yields that $$\frac{F\left(\left(X^c_t+v\one_{[0,t]}(r)\right)_{t\geq 0}\right)-F\left(\left(X^c_t\right)_{t\geq 0}\right)}{v}$$ and $$\frac{F\left(\left(X_t+v\one_{[0,t]}(r)\right)_{t\geq 0}\right)-F\left(\left(X_t\right)_{t\geq 0}\right)}{v}$$
have the same integrands in their chaos expansion. Since $\D Y^c$ and $\D Y$ have the same integrands in their chaos expansion as well, the assertion follows.
\end{proof}
The next theorem proposes a chain rule-like application using the results so far. For this we rely on Theorem \ref{repthm} which justifies the existence of a measurable functional $G\colon D\left({[0,\infty[}\right)\times\R\to\R$ such that a $\mathcal{F}\otimes\mathcal{B}(\R)$-measurable function $f\colon\Omega\times\R\to\R$ can be represented by
\begin{equation}\label{repf}
x\mapsto f(\omega,x)= x\mapsto G(\left(X_t(\omega)\right)_{t\geq 0},x),\quad\mathbb{P}\text{-a.s.}
\end{equation}
\begin{thm}
Assume a $\mathcal{F}\otimes\mathcal{B}(\R)$-measurable process $f\colon\Omega\times\R\to\R$ such that for all $x\in\R$, the random variable
$\omega\mapsto f(\omega,x)$ is in $\mathbb{D}_{1,2}(\Omega)$ and let $G$ be as in (\ref{repf}). Let $Y$ be a random variable in $\mathbb{D}_{1,2}(\Omega)$, and assume that $f(\cdot,Y)\in\mathbb{D}_{1,2}(\Omega)$. Then the equation
\begin{equation*}
\begin{split}
\D_{r,v}f(\cdot,Y)(\omega)=&\left(\D_{r,v}f\right)\left(\omega,Y(\omega)+v\D_{r,v}Y(\omega)\right)\\
&+\frac{f\left(\omega,Y(\omega)+v\D_{r,v}Y(\omega)\right)-f(\omega,Y(\omega))}{v}
\end{split}
\end{equation*}
holds for $\mathbb{P}\otimes\m$-almost all $(\omega,r,v)\in\Omega\times{[0,\infty[}\times\left(\R\setminus\{0\}\right)$, where $\left(\D_{r,v}f\right)(\omega,x)$ denotes
\begin{equation}\label{fdefg}
\frac{G\left(\left(X_t(\omega)+v\one_{[0,t]}(r)\right)_{t\geq 0},x\right)-G\left(\left(X_t(\omega)\right)_{t\geq 0},x\right)}{v}.
\end{equation}
\end{thm}
\begin{proof}
Since we have by Lemma \ref{replem} that $Y=F(X)$, $\mathbb{P}$-a.s., we can use equation (\ref{repf}) to get $$f(\cdot,Y)=G(X,F(X)),\ \mathbb{P}\text{-a.s.}$$ The measurability of the mapping $$D\left({[0,\infty[}\right)\to\R,\ h\mapsto G(h,F(h))$$ and Theorem \ref{diffthm} imply that $\mathbb{P}\otimes\m$-a.e.
\begin{equation*}
\begin{split}
&\D_{r,v}f(\cdot,Y)\\
&=\frac{G\left(\left(X_t+v\one_{[0,t]}(r)\right)_{t\geq 0},F\left(\left(X_t+v\one_{[0,t]}(r)\right)_{t\geq 0}\right)\right)-G\left(X,F\left(X\right)\right)}{v}.
\end{split}
\end{equation*}
Splitting up the last term into two summands, using (\ref{chain}) to get $$F\left(\left(X_t+v\one_{[0,t]}(r)\right)_{t\geq 0}\right)=Y+v\D_{r,v}Y,\quad\mathbb{P}\otimes\m\text{-a.e.},$$  we arrive at
\begin{equation*}
\begin{split}
&\frac{G\left(\left(X_t+v\one_{[0,t]}(r)\right)_{t\geq 0},Y+v\D_{r,v}Y\right)-G\left(X,Y+v\D_{r,v}Y\right)}{v}\\
&+\frac{G\left(X,Y+v\D_{r,v}Y\right)-G\left(X,Y\right)}{v}\\
&=\left(\D_{r,v}f\right)\left(\cdot,Y+v\D_{r,v}Y\right)+\frac{f\left(\cdot,Y+v\D_{r,v}Y\right)-f(\cdot,Y)}{v},
\end{split}
\end{equation*}
$\mathbb{P}\otimes\m$-a.e.
\end{proof}
\begin{rem}
\begin{enumerate}
\item[(i)] The random variables $f(\cdot,Y)$, $G\left(X,Y+v\D_{r,v}Y\right)$ etc. were defined as usual by first selecting representatives of $\tilde Y\in Y$ and $Z\in Y+v\D_{r,v}Y$, then taking the equivalence classes of $f(\cdot, \tilde Y)$ and $G(X,Z)$ in $\mathrm{L}_0\left(\Omega,\mathcal{F},\mathbb{P}\right)$ and $\mathrm{L}_0\left(\mathbb{P}\otimes\m\right)$, respectively. Equation (\ref{repf}) implies that the definition of $G\left(X,Y+v\D_{r,v}Y\right)$ is meaningful and does not depend on the choice of the functional $G$ up to functions being zero $\mathbb{P}\otimes\m$-a.e.
\item[(ii)] The expression $\left(\D_{r,v}f\right)\left(\cdot,Y+v\D_{r,v}Y\right)$ in $\mathrm{L}_0(\mathbb{P}\otimes\m)$ in the sense of (i) and (\ref{fdefg}) is well-defined, i.e. it does not depend on the choice of the functional $G$:

By the same calculations as in the proof, for another functional $\tilde G$ satisfying (\ref{repf}), one gets $\mathbb{P}\otimes\m$-a.e.
\begin{equation}\label{eqpartial}
\begin{split}
\D_{r,v}f(\cdot,Y)=&\frac{\tilde G\left(\!\left(X_t+v\one_{[0,t]}(r)\right)_{t\geq 0},Y\!\!+v\D_{r,v}Y\right)\!-\!\tilde G\left(X,Y\!\!+v\D_{r,v}Y\right)}{v}\\
&-\frac{\tilde G\left(X,Y+v\D_{r,v}Y\right)-\tilde G\left(X,Y\right)}{v}.
\end{split}
\end{equation}
Considering (i) and (\ref{repf}) we conclude that $$\tilde G\left(X,Y+v\D_{r,v}Y\right)=G\left(X,Y+v\D_{r,v}Y\right),\quad\mathbb{P}\otimes\m\text{-a.e.,}$$ therefore also
\begin{equation}\label{eqpartial2}
\begin{split}&\tilde G\left(\left(X_t+v\one_{[0,t]}(r)\right)_{t\geq 0},Y+v\D_{r,v}Y\right)\\
&= G\left(\left(X_t+v\one_{[0,t]}(r)\right)_{t\geq 0},Y+v\D_{r,v}Y\right),\quad \mathbb{P}\otimes\m\text{-a.e.},
\end{split}
\end{equation}
because we can express these terms with help of equation (\ref{eqpartial}) for $G$ and $\tilde G$. Equation (\ref{eqpartial2}) then shows that $\left(\D_{r,v}f\right)\left(\cdot,Y+v\D_{r,v}Y\right)$ is uniquely defined up to functions which equal zero $\mathbb{P}\otimes\m$-a.e.
\end{enumerate}
\end{rem}
\bibliographystyle{plain}

\end{document}